\documentclass[12pt,reqno]{amsart}
\usepackage{amssymb}
\usepackage{eurosym}
\usepackage{epsfig}
\usepackage{amsmath}
\usepackage{amsfonts}
\usepackage{pgf,tikz}
\usepackage{enumerate}
\usepackage{cite}
\usepackage[colorlinks=true, linkcolor=black, citecolor=black, urlcolor=black]{hyperref}

 \usepackage{color}

\newtheorem{theorem}{\sc Theorem}[section]

\newtheorem{remark}[theorem]{\sc Remark}

\newtheorem{example}[theorem]{\sc Example}

\def\qed{\hbox to 0pt{}\hfill$\rlap{$\sqcap$}\sqcup$\medbreak}

\numberwithin{equation}{section}

\begin{document}

\title[Birkhoff--Kellogg type results]{Birkhoff--Kellogg type results in product spaces and their application to differential systems}

\author[A. Calamai]{Alessandro Calamai}
\address{Alessandro Calamai, 
Dipartimento di Ingegneria Civile, Edile e Architettura,
Universit\`{a} Politecnica delle Marche
Via Brecce Bianche
I-60131 Ancona, Italy}%
\email{a.calamai@univpm.it}%

	\author[G. Infante]{Gennaro Infante}
	\address{Gennaro Infante, Dipartimento di Matematica e Informatica, Universit\`{a} della
		Calabria, 87036 Arcavacata di Rende, Cosenza, Italy}%
	\email{gennaro.infante@unical.it}%
	
	\author[J. Rodr\'iguez--L\'opez]{Jorge Rodr\'iguez--L\'opez}
	\address{Jorge Rodr\'iguez--L\'opez, CITMAga \& Departamento de Estat\'{\i}stica, An\'alise Matem\'atica e Optimizaci\'on, Universidade de Santiago de Compostela,  15782, Facultade de Matem\'aticas, Campus Vida, Santiago, Spain}
	\email{jorgerodriguez.lopez@usc.es}%

\date{}

\begin{abstract} 
We provide a new version of the well-known Birkhoff-Kellogg invariant-direction Theorem in product spaces.
Our results concern operator systems and give the existence of component-wise eigenvalues, instead of scalar eigenvalues as in the classical case, 
that have corresponding eigenvectors with all components nontrivial and localized by their norm.
We also show that, when applied to nonlinear eigenvalue problems for differential equations,
this localization property of the eigenvectors provides, in turn, qualitative properties of the solutions.
This is illustrated in two contexts of systems of PDEs and ODEs. We show the applicability of our theoretical results with two explicit examples.
\end{abstract}



 \maketitle

\section{Introduction}
A celebrated result in Nonlinear Analysis is the Birkhoff-Kellogg invariant-direction Theorem~\cite{B-K-1922}.
This theorem has been object of extensive research in the past 
and finds significant applications in the study of nonlinear eigenvalue problems in infinite-dimensional normed linear spaces, see for example the books~\cite{ADV, guolak, Krasno}, the recent papers \cite{CI_MMAS, gi-BK}, and references therein.
In the version by Krasnosel'ski\u{\i} and Lady\v{z}enski\u{\i}~\cite{Kra-Lady}, 
a similar result is set in cones of real Banach spaces, yielding the existence of a pair, constituted by a positive eigenvalue and an eigenvector, the latter localized inside a cone.
We stress that a notable, common feature of the two above mentioned theorems is that they provide a localization of the eigenfunction; this localization, in turn, provides qualitative properties of the solution in the context of applications to nonlinear eigenvalue problems for differential equations.
In the framework of systems the situation is somewhat more delicate. In fact, a direct application of one of the previous two results in product spaces would provide the existence of an eigenvalue with a corresponding (vectorial) eigenfunction that may have trivial components; this issue has been been highlighted for example in~\cite[Definition 1]{gi-BK}.

In this paper we present a new version of the Birkhoff--Kellogg Theorem in product spaces, which, instead of a scalar eigenvalue, provides the existence of \textit{component-wise} eigenvalues
 that have corresponding eigenvectors with all components nontrivial and localized by their norm.
 In the context of  systems of integral equations and their applications, component-wise eigenvalues have been investigated in Chapter 3 of the book~\cite{ROW}, where the authors sought constant-sign eigenvectors by means of topological tools such as the Schauder and  the  Krasnosel'ski\u{\i}-Guo fixed point theorems, and in the book~\cite{HL} where the authors sought, also via topological fixed point theory, 
the existence of  positive eigenvectors. More in details,  we study systems of type
	\begin{equation*} \label{sys_intro}
		\left\{\begin{array}{l} x=\lambda_1\, T_1(x,y), \\ y=\lambda_2\,T_2(x,y), \end{array} \right.
	\end{equation*}
where $T=(T_1,T_2)$ is a suitable compact operator acting on the Cartesian product of two sets $C_1$ and $C_2$, which can be as follows:
\begin{enumerate}
	\item $C_1$ and $C_2$ are cones; \label{uno}
	\item $C_1$ is a cone and $C_2=X_2$ is a infinite dimensional normed space; \label{due}
	\item both $C_1=X_1$ and $C_2=X_2$ are infinite dimensional normed spaces. \label{tre}
\end{enumerate}

In the context of fixed point theory, a component-wise approach has been utilized in the past, see for example~\cite{Avramescu, Benedetti, imap, InMaRo, Perov, PrecupFPT}; here we develop a somewhat analogue theory in the framework of nonlinear spectral theory. In particular, 
  we give quite natural conditions yielding the existence of component-wise eigenvalues $\lambda_1,\lambda_2>0$ and corresponding eigenfunctions $x_0,y_0$, both of prescribed non-zero norm.
We also provide further results in the settings~\eqref{due} and \eqref{tre}, which yield the additional existence of negative eigenvalues with corresponding eigenfunctions. Furthermore we present a direct application of the classical 
result by Krasnosel'ski\u{\i} and Lady\v{z}enski\u{\i} to the context of systems and present a comparison between the two approaches, see Remark~\ref{comparison}. 

Section 3 of the paper is devoted to applications of the theoretical results to differential systems. We stress that 
these applications are given as a motivation for the theoretical results and are provided for illustrative purposes. In order to keep the paper as readable as possible, we study these problems in the case of continuous nonlinearities.

We firstly focus on the systems of PDEs 
\begin{equation}\label{sys_elliptic-intro}
	\left\{\begin{array}{ll} -\Delta u=\lambda_1 f(x,u,v), & \ \text{ in } \Omega, \\ -\Delta v=\lambda_2 g(x,u,v), & \ \text{ in } \Omega, \\ u=v=0, & \ \text{ on } \partial\,\Omega,  \end{array} \right.
\end{equation}
where $\Omega\subset\mathbb{R}^n$ denotes the open unit ball in $\mathbb{R}^n$ and $f,g$ are suitable non-negative continuous functions. 
Problems of this type are well-studied, both in the classical case of the spectral problem, that is with $\lambda=\lambda_1=\lambda_2$,
see for example~\cite{CCI14, cui3, imap, ma2},
and in the case of component-wise eigenvalues with possibly different $\lambda_1,\lambda_2$,
see for instance~\cite{chzh, hai, lan2, lan2021} and references therein. 
Within this framework, we utilize the cone product setting \eqref{uno}; in Theorem~\ref{th_PDE_exis}
we provide sufficient conditions for \eqref{sys_elliptic-intro} that ensure the existence of a pair of positive 
component-wise eigenvalues with eigenfunctions possessing nontrivial components with localized norms.
The result is then illustrated in a specific example.

Our second setting of application is   the following BVP for a system of ODEs
\begin{equation}\label{syst_ode-intro}
\left\{	\begin{array}{l}
		u''(t)+\lambda_1 f(t,u(t),v(t))=0, \ t\in[0,1], \\
		v''(t)+\lambda_2 g(t,u(t),v(t))=0, \ t\in[0,1],\\
		u'(0)=u(1)+u'(1)=0,\\ v'(0)=v(1)-\frac12 v'(1)=0.	
\end{array}
\right.
\end{equation}
where $f,g$ are suitable non-negative continuous functions. The BCs that occur in~\eqref{syst_ode-intro} have been investigated for a different set of parameters by Lan~\cite{Lan, lan2021}. 
We rewrite the system~\eqref{syst_ode-intro} in terms of a system of Hammerstein integral equations. Since one of the two corresponding kernels changes sign,  we use the setting \eqref{due},
namely the Cartesian product of a conical shell times a ball in the space of continuous  functions.
In this case our approach yields the existence of two distinct pairs of component-wise eigenvalues with nontrivial eigenfunctions possessing nontrivial components;
this is illustrated in a toy model as well.

Overall results are new, from both the theoretical and the applied point of view, and complement the ones in~\cite{ROW, ADV, CCI14, chzh, cui3, hai, HL, gi-BK, imap, lan2, lan2021, ma2, PrecupFPT}.

\section{Birkhoff-Kellogg type results}
We begin this Section by recalling the classical Birkhoff-Kellogg invariant-direction Theorem~\cite{B-K-1922}, cf.~\cite[Theorem~6.1]{GD}.

\begin{theorem}%[Birkhoff-Kellogg]%\label{B-K-1922}
Let $U$ be a bounded open neighbourhood of $0$ in an infinite-dimensional normed linear space $(V, \|\ \|)$, and let $T:\partial U \to V$ be a compact map satisfying $\|T(x)\| \geq \alpha$ for some $\alpha > 0$ for every $x$ in $\partial U$. Then there exist $x_0 \in \partial U$ and $ \lambda_{0} \in (0,+\infty)$ such that 
$x_0 = \lambda_{0} T(x_0).$
\end{theorem}

In the following 
version by Krasnosel'ski\u{\i} and Lady\v{z}enski\u{\i}~\cite{Kra-Lady}, cf.~\cite[Theorem~5.5]{Krasno}, 
a similar result is set in cones of real Banach spaces;
we recall that a cone $K$ of a normed linear space
 $(X,\| \, \|)$ is a closed set with $K+K\subset K$, $\mu K\subset K$ for all $\mu\ge 0$ and $K\cap(-K)=\{0\}$.

	\begin{theorem}\label{th_KL}
		Let $X$ be a real Banach space, $U\subset X$ be an open bounded set with $0\in U$, $K\subset X$ be a cone, $T:K\cap\overline{U}\longrightarrow K$ be compact and suppose that
		\[\inf_{x\in K\cap\partial\,U}\left\|Tx\right\|>0. \]
		Then there exist $x_0\in K\cap\partial\,U$ and $\lambda_0>0$ such that $x_0=\lambda_0\, Tx_0$.
	\end{theorem}	

\medskip
Before stating our results we fix some notation. 
Let $(X,\| \, \|)$ be a normed linear space and
$K$ a cone in $X$.
Given $r>0$, by $B_{r}$ we mean the open ball in $X$ centered at the origin and with radius $r$, while
by $\overline{B}_{r}$, $\partial\overline{B}_{r}$ we mean the closed disk and its boundary, respectively.
Moreover, we denote by $K_{r}= B_{r}\cap K$, and by $\overline{K}_{r}= \overline{B}_{r} \cap K$, resp.\ 
$\partial\overline{K}_{r}= \partial\overline{B}_{r} \cap K$, the closure and boundary of $K_{r}$ relative to $K$.

Observe that $\partial  \overline{K}_{r}$ is a retract of $\overline{K}_{r}$. An explicit example of such a retraction can be found in \cite[Example 3]{fel} defined as
\[\rho(z)=r\dfrac{z+(r-\left\|z\right\|)^2 h}{\left\|z+(r-\left\|z\right\|)^2 h\right\|}, \quad z\in \overline{K}_{r}, \]
where $h\in K\setminus\{0 \}$ is fixed.

With abuse of notation (the whole space $X$, when nontrivial, is not a cone) we will still denote $X_{r}= B_{r}$, $\overline{X}_{r}= \overline{B}_{r}$,  
$\partial\overline{X}_{r}= \partial\overline{B}_{r}$, so that, if $X$ is infinite dimensional, again $\partial  \overline{X}_{r}$ is a retract of $\overline{X}_{r}$.

Let $X_1$ and $X_2$ be normed linear spaces and $C_1\subset X_1$, $C_2\subset X_2$ such that for each $i\in\{1,2\}$ either
\begin{enumerate}
	\item[(a)] $C_i=K_i$ is a cone; or
	\item[(b)] $C_i = X_i$ is an infinite dimensional normed space. 
\end{enumerate}

In the following we will consider the product space $X_1 \times X_2$ endowed with the maximum norm.
With abuse of notation, the norm in any space will be still denoted by $\|\ \|$.
The next result is a version of the Birkhoff--Kellogg Theorem in product spaces.

\begin{theorem}\label{th_sys}
	Let $r_1,r_2$ be positive constants and suppose that $$T=(T_1,T_2):\overline{C}_{1,r_1}\times \overline{C}_{2,r_2}\longrightarrow C_1\times C_2$$ is a compact map satisfying that 
		\begin{equation}\label{cond_BK}
\inf_{\left\|x\right\|=r_1,\ \left\|y\right\|= r_2}\left\|T_1(x,y)\right\|>0 \quad \text{and} \quad \inf_{\left\|x\right\|= r_1,\ \left\|y\right\|= r_2}\left\|T_2(x,y)\right\|>0.
	\end{equation}
	Then there exist $\lambda_1,\lambda_2>0$ and $(x_0,y_0)\in \overline{C}_{1,r_1}\times \overline{C}_{2,r_2}$ with $\left\|x_0\right\|=r_1$ and $\left\|y_0\right\|=r_2$ such that
	\begin{equation}\label{eq_sol_sys}
		\left\{\begin{array}{l} x_0=\lambda_1\, T_1(x_0,y_0), \\ y_0=\lambda_2\,T_2(x_0,y_0). \end{array} \right.
	\end{equation}
\end{theorem}

\begin{proof}
For each $i=1,2$, let us consider a retraction $\rho_i:\overline{C}_{i,r_i}\rightarrow  \partial  \overline{C}_{i,r_i}$. 
Now, define the auxiliary map $N=(N_1,N_2): \overline{C}_{1,r_1}\times \overline{C}_{2,r_2} \rightarrow \overline{C}_{1,r_1}\times \overline{C}_{2,r_2}$ as
\[N(x,y)=\left(r_1\dfrac{T_1\left(\rho_1(x),\rho_2(y)\right)}{\left\|T_1\left(\rho_1(x),\rho_2(y)\right)\right\|},r_2\dfrac{T_2\left(\rho_1(x),\rho_2(y)\right)}{\left\|T_2\left(\rho_1(x),\rho_2(y)\right)\right\|} \right), \]
and observe that, by \eqref{cond_BK}, $N$ is well-defined.	
Since $N$ is a compact map, Schauder fixed point theorem (see e.g.\ \cite[\S 6, Theorem~3.2]{GD})
 ensures that $N$ has at least one fixed point $(x_0,y_0)\in \overline{C}_{1,r_1}\times \overline{C}_{2,r_2}$. Observe that $N\left( \overline{C}_{1,r_1}\times \overline{C}_{2,r_2} \right)\subset  \partial \overline{C}_{1,r_1}\times \partial \overline{C}_{2,r_2}$ and so it follows that $(x_0,y_0)\in \partial \overline{C}_{1,r_1}\times \partial \overline{C}_{2,r_2}$, that is,
\[x_0=r_1\dfrac{T_1(x_0,y_0)}{\left\|T_1(x_0,y_0)\right\|}, \quad y_0=r_2\dfrac{T_2(x_0,y_0)}{\left\|T_2(x_0,y_0)\right\|}. \]
Taking $\lambda_i=r_i/\left\|T_i(x_0,y_0)\right\|$, $i=1,2$, the proof is finished.
\end{proof}

\begin{remark}
	It should be noted that, under the assumptions of Theorem~\ref{th_sys}, the existence of $\lambda_0>0$ and $(x_0,y_0)\in \overline{C}_{1,r_1}\times \overline{C}_{2,r_2}$ with $\left\|x_0\right\|=r_1$ and $\left\|y_0\right\|=r_2$ solving
    the equation
   	\begin{equation*}\label{sys_1par}  
		\left\{\begin{array}{l} x_0=\lambda_0\, T_1(x_0,y_0), \\ y_0=\lambda_0\,T_2(x_0,y_0) \end{array} \right.
	\end{equation*}
    cannot be guaranteed.
	Indeed, consider as Banach spaces $X_1=X_2=\mathbb{R}$,
	the cones $K_1=K_2=[0,+\infty)$, $r_1=r_2=1$ and the continuous function $T:[0,1]\times[0,1]\rightarrow [0,+\infty)\times [0,+\infty)$ given by
	\[T(x,y)=(T_1(x,y),T_2(x,y))=(2x+y,x+3y). \]
	Note that 
	$\displaystyle \inf_{{\left\|x\right\|=r_1,\ \left\|y\right\|= r_2}}\left|T_1(x,y)\right|=3>0$ and $\displaystyle \inf_{{\left\|x\right\|=r_1,\ \left\|y\right\|= r_2}}\left|T_2(x,y)\right|=4>0$, 
	but there is no $\lambda\in(0,+\infty)$ such that $(1,1)=\lambda\,T(1,1)=\lambda\,(3,4)$.
\end{remark}

Let us focus now on operators defined in the product of a cone times an infinite dimensional normed space. In this case, an additional solution can be obtained. 

\begin{theorem} \label{thmsys-kb}
	Let $K_1$ be a cone in the normed linear space $X_1$, and $X_2$ be an infinite dimensional normed space. 
	Let $r_1,r_2$ be positive constants and suppose that $T=(T_1,T_2):\overline{K}_{1,r_1}\times \overline{B}_{2,r_2}\longrightarrow K_1\times X_2$ is a compact map satisfying that 
	\[\inf_{\left\|x\right\|= r_1,\ \left\|y\right\|= r_2}\left\|T_1(x,y)\right\|>0 \quad \text{and} \quad \inf_{\left\|x\right\|= r_1,\ \left\|y\right\|= r_2}\left\|T_2(x,y)\right\|>0. \]
	Then there exist $\lambda_{1,1},\lambda_{2,1},\lambda_{1,2} >0$, $\lambda_{2,2}<0$ and $(x_{0,j},y_{0,j})\in \overline{K}_{1,r_1}\times \overline{B}_{2,r_2} $, $j=1,2$, with $\left\|x_{0,j}\right\|=r_1$ and $\left\| y_{0,j}\right\|=r_2$ such that
	\begin{equation*}
		\left\{\begin{array}{l} x_{0,j}=\lambda_{1,j}\, T_1(x_{0,j},y_{0,j}), \\ y_{0,j}=\lambda_{2,j}\,T_2(x_{0,j},y_{0,j}), \end{array} \right. \qquad (j=1,2).
	\end{equation*}
\end{theorem} 

\begin{proof}
	The first solution is ensured by Theorem \ref{th_sys}. In order to obtain the second one, just consider the map $\tilde{N}: \overline{K}_{1,r_1}\times \overline{B}_{2,r_2}\rightarrow \overline{K}_{1,r_1}\times \overline{B}_{2,r_2}$ given by
	\[\tilde{N}(x,y)=\left(r_1\dfrac{T_1\left(\rho_1(x),\rho_2(y)\right)}{\left\|T_1\left(\rho_1(x),\rho_2(y)\right)\right\|},-r_2\dfrac{T_2\left(\rho_1(x),\rho_2(y)\right)}{\left\|T_2\left(\rho_1(x),\rho_2(y)\right)\right\|} \right). \]
	As a consequence of Schauder fixed point theorem, $\tilde{N}$ has a fixed point $(x_{0,2},y_{0,2})$ located in $\partial \overline{K}_{1,r_1} \times \partial\overline{B}_{2,r_2}$, that is, 
	\[\left\{\begin{array}{l} x_{0,2}=\lambda_{1,2}\, T_1(x_{0,2},y_{0,2}), \\ y_{0,2}=\lambda_{2,2}\,T_2(x_{0,2},y_{0,2}), \end{array} \right. \]
	where $\lambda_{1,2}=r_1/\left\|T_1\left(x_{0,2},y_{0,2}\right)\right\|>0$ and $\lambda_{2,2}=-r_2/\left\|T_2\left(x_{0,2},y_{0,2}\right)\right\|<0$.
\end{proof}

\begin{remark}
	Under the assumptions of Theorem \ref{th_sys}, if both $C_1$ and $C_2$ are infinite dimensional normed spaces, then there exist four couples of numbers $\lambda_1,\lambda_2$ and associated points $(x_0,y_0)\in \overline{B}_{1,r_1}\times \overline{B}_{2,r_2} $ with $\left\|x_0\right\|=r_1$ and $\left\|y_0\right\|=r_2$ such that \eqref{eq_sol_sys} holds. Indeed, it suffices to apply the Schauder theorem to each auxiliary map
	\[N_{j,k}(x,y)=\left((-1)^j r_1\dfrac{T_1\left(\rho_1(x),\rho_2(y)\right)}{\left\|T_1\left(\rho_1(x),\rho_2(y)\right)\right\|},(-1)^k r_2\dfrac{T_2\left(\rho_1(x),\rho_2(y)\right)}{\left\|T_2\left(\rho_1(x),\rho_2(y)\right)\right\|} \right), \quad j,k\in\{1,2\}. \]
\end{remark}

Let us now stress, once again, that a key feature of our results is the localization of the solutions.
In order to compare our results with the classical ones we provide now a direct application of Theorem \ref{th_KL} to 
the situation in which, for each $i\in\{1,2\}$, $C_i=K_i$ is a cone. An analogous statement holds when $C_i$, $i\in\{1,2\}$,  is either a cone or an infinite dimensional normed space.

\begin{theorem}\label{th_BK_cones}
	Let $X_1$ and $X_2$ be real Banach Spaces and let $K_1\subset X_1$, $K_2\subset X_2$ cones.
    Let $r_1,r_2$ be positive constants and suppose that 
    $$
    T=(T_1,T_2):\overline{K}_{1,r_1}\times \overline{K}_{2,r_2}\longrightarrow K_1\times K_2,
    $$ is a compact map satisfying that 
		\begin{equation}\label{cond_BK_old}
\inf_{\left\|x\right\|=r_1,\ \left\|y\right\| \leq r_2}\left\|T_1(x,y)\right\|>0 \quad \text{and} \quad \inf_{\left\|x\right\|\leq r_1,\ \left\|y\right\|= r_2}\left\|T_2(x,y)\right\|>0.
	\end{equation}
	Then there exists $\lambda_0>0$ and $(x_0,y_0)\in \overline{K}_{1,r_1}\times \overline{K}_{2,r_2}$ with $\left\|x_0\right\|=r_1$ or $\left\|y_0\right\|=r_2$ such that
	\begin{equation}\label{eq_sol_sys_old}
		\left\{\begin{array}{l} x_0=\lambda_0\, T_1(x_0,y_0), \\ y_0=\lambda_0\,T_2(x_0,y_0). \end{array} \right.
	\end{equation}
\end{theorem}

\begin{proof}
First of all note that $K=K_1\times K_2$ is a cone in the space $X=X_1\times X_2$. Let $U={B}_{1,r_1}\times {B}_{2,r_2}$ and take $(x,y) \in K\cap\partial U$. Now either $(x,y) \in \partial B_{1,r_1} \times \overline{B}_{2,r_2}$ or $(x,y)\in \overline{B}_{1,r_1} \times \partial B_{2,r_2}$. 
In the first case we have 
$$ \|T(x,y)\| \geq  \|T_1(x,y)\|$$ while in the second case we obtain $$ \|T(x,y)\| \geq  \|T_2(x,y)\|.$$ In both cases, by condition~\eqref{cond_BK_old}, we have
\begin{multline*}
    \|T(x,y)\| \geq \min\{ \|T_1(x,y)\|, \|T_2(x,y)\|\}\\ \geq  \min \Bigl \{ \inf_{\left\|x\right\| =r_1,\ \left\|y\right\| \leq r_2}\left\|T_1(x,y)\right\|, \inf_{\left\|x\right\|\leq r_1,\ \left\|y\right\|= r_2}\left\|T_2(x,y)\right\| \Bigr \}>0.
\end{multline*}
This
implies the inequality
    $$ \inf_{(x,y) \in K\cap\partial U} \|T(x,y)\| > 0. $$
    Thus, by Theorem \ref{th_KL}, there exist $\lambda_0 > 0$ and $(x_0, y_0) \in K\cap\partial U$ such that \eqref{eq_sol_sys_old} holds.
\end{proof}
\begin{remark}\label{comparison}
Note that condition~\eqref{cond_BK_old} is stronger than~\eqref{cond_BK}; this situation is particularly meaningful in the applications, see the hypothesis  $b)$ in Theorem~\ref{th_ODE_exis}.
Furthermore, note that Theorem~\ref{th_sys} provides a better localization of the eigenfunctions (both of them are nontrivial), since Theorem~\ref{th_BK_cones} provides couples  $(x_0, y_0)$ in which one of the components can be trivial.
\end{remark}

\section{Some applications to differential systems} \label{sec_appl}
In this Section we present some applications of the above results to the context of systems of ODEs and PDEs.

\subsection{Eigenvalues for a system of elliptic PDEs}\label{eigen1}
We begin by illustrating the applicability of Theorem~\ref{th_sys} in the context of PDEs. In particular, we discuss the existence of eigenvalues and eigenfunctions 
of quasilinear elliptic systems subject to homogeneous Dirichlet boundary conditions of the form
\begin{equation}\label{sys_elliptic}
	\left\{\begin{array}{ll} -\Delta u=\lambda_1 f(x,u,v), & \ \text{ in } \Omega, \\ -\Delta v=\lambda_2 g(x,u,v), & \ \text{ in } \Omega, \\ u=v=0, & \ \text{ on } \partial\,\Omega,  \end{array} \right.
\end{equation}
where $\Omega\subset\mathbb{R}^n$ denotes the open unit ball in $\mathbb{R}^n$, $f:\overline{\Omega}\times \mathbb{R}_{+}\times\mathbb{R}_{+}\rightarrow\mathbb{R}_{+}$ and $g:\overline{\Omega}\times \mathbb{R}_{+}\times\mathbb{R}_{+}\rightarrow \mathbb{R}_{+}$ are continuous functions.

It is folklore (see for example~\cite[Section 7.2]{Krasno} or \cite[Section 4.2]{GT}), that if $h$ is sufficiently smooth, the unique solution of the linear PDE 
\begin{equation}\label{linearbvp}
-\Delta w =h(x) \ \text{ in } \Omega, \quad w=0 \ \text{ on } \partial\,\Omega, 
\end{equation}
can be represented 
in the integral form
\begin{equation}\label{linearHam}
w(x)=\displaystyle \int_{\Omega}k(x,y)\,h(y)\,dy=:L(h)(x),
\end{equation}
where $k$ is the Green's function associated to the PDE. The explicit formula for the Geeen's function for the ball depends on the dimension of the space and can be found for example in \cite[Section 2]{GT}. It is also well known that the Green's function is non-negative and continuous for $x\neq y$ and singular on the diagonal. The recent paper~\cite{yang2021} provides the expression of the Green's function and studies its properties in more general domains. Note that if the function $h$ that occurs in~\eqref{linearbvp} is continuous one may still study the solutions of the corresponding Hammerstein integral equation~\eqref{linearHam} and may interpret those as \emph{weak} solutions of~\eqref{linearbvp}, see~\cite[Section 7.2]{Krasno}. 
Now, denote by 
$$P:=\{u\in \mathcal{C}(\overline{\Omega}):u\geq 0  \},$$
where $\mathcal{C}(\overline{\Omega})$ is the space  of continuous functions endowed with the usual supremum norm
$\left\|u\right\|_{\infty}=\max_{x\in\overline{\Omega}}\left|u(x) \right|$. Note that $P$ is a cone in $\mathcal{C}(\overline{\Omega})$ and 
the linear operator $L$ leaves the cone $P$ invariant and is compact, see for example~\cite[Section 7.2]{Krasno} and~\cite[Proposition~5.1]{yang2021}. 

With these ingredients in mind, we say that a pair $(u,v)\in \mathcal{C}(\overline{\Omega})\times \mathcal{C}(\overline{\Omega})$ is a {\it (weak) solution}
of problem~\eqref{sys_elliptic} if $(u,v)$ is a solution of the integral system 
\begin{equation*}\label{sys_Ham-PDE}
	\left\{\begin{array}{l} 
	u(x)=\displaystyle \lambda_1 \int_{\Omega}k(x,y)\,f(y,u(y),v(y))\,dy, \\[0.3cm] 
	v(x)=\displaystyle\lambda_2 \int_{\Omega}k(x,y)\,g(y,u(y),v(y))\,dy.  
	\end{array} \right.
\end{equation*}
In other words, $(u,v)$ is a solution of \eqref{sys_elliptic}
if and only if
\begin{equation}\label{sys_abstract}
	\left\{\begin{array}{l} 
	u=\displaystyle \lambda_1 T_1(u,v), \\[0.3cm] 
	v=\displaystyle\lambda_2 T_2(u,v),  
	\end{array} \right.
\end{equation}
where 
\begin{equation*}\label{opTpde}
	\begin{array}{l} 
	T_1(u,v)(x)=\displaystyle  \int_{\Omega}k(x,y)\,f(y,u(y),v(y))\,dy, \\[0.3cm] 
	T_2(u,v)(x)=\displaystyle \int_{\Omega}k(x,y)\,g(y,u(y),v(y))\,dy.  
	\end{array} 
\end{equation*}
We are now in position to state the following Theorem.
\begin{theorem}\label{th_PDE_exis}
Let $r_1,r_2$ be positive constants and suppose that exist two continuous functions $\underline{f}, \underline{g}:\overline{\Omega}\rightarrow\mathbb{R}_+$ such that the following conditions hold:
	\begin{enumerate}[$a)$]
		\item $ f(x,u,v)\geq \underline{f}(x)$ on $\overline{\Omega}\times [0,r_1]\times[0,r_2]$ and 
		\[\sup_{x\in\overline{\Omega}}\int_{\Omega}k(x,y)\,\underline{f}(y)\,dy>0;  \]
		\item $g(x,u,v)\geq \underline{g}(x)$ on $\overline{\Omega}\times [0,r_1]\times[0,r_2]$
		and 
		\[\sup_{x\in\overline{\Omega}}\int_{\Omega}k(x,y)\,\underline{g}(y)\,dy>0.  \]
	\end{enumerate}	
Then there exist $\lambda_1,\lambda_2>0$ and $(u_0,v_0)\in P\times P$ with $\left\|u_0\right\|_{\infty}=r_1$ and $\left\|v_0\right\|_{\infty}=r_2$ that satisfy the system~\eqref{sys_abstract}.
\end{theorem}

\begin{proof}
Let us consider the Banach spaces  $X_1=X_2= \mathcal{C}(\overline{\Omega})$ and the cones $K_1=K_2=P$. 
Due to the compactness and invariance properties of the operator $L$ defined in~\eqref{linearHam}, combined with the continuity of $f,g$, the operator
$$T=(T_1,T_2):\overline{P}_{r_1}\times \overline{P}_{r_2}\longrightarrow P\times P,$$
defined in \eqref{opTpde},
is compact. Now let us consider $(u,v)\in P\times P$ such that $\left\|u\right\|_{\infty}=r_1, \left\|v\right\|_{\infty}= r_2$. For every $x\in \overline{\Omega}$ we have
\begin{equation*}
\|T_1(u,v)\|_{\infty} \geq T_1(u,v)(x)=\displaystyle  \int_{\Omega}k(x,y)\,f(y,u(y),v(y))\,dy\geq \displaystyle  \int_{\Omega}k(x,y)\,\underline{f}(y)\,dy.
\end{equation*}
Then we get
\begin{equation}\label{RHS_PDE}
\|T_1(u,v)\|_{\infty} \geq \sup_{x\in\overline{\Omega}}  \int_{\Omega}k(x,y)\,\underline{f}(y)\,dy.
\end{equation}
Note that the RHS of \eqref{RHS_PDE} does not depend on the particular $(u,v)$ chosen. Therefore we  obtain
$$
\inf_{\left\|u\right\|_{\infty}=r_1,\  \left\|v\right\|_{\infty}= r_2}\left\|T_1(u,v)\right\|_{\infty} \geq \sup_{x\in\overline{\Omega}}  \int_{\Omega}k(x,y)\,\underline{f}(y)\,dy >0.
$$
A similar argument applies in the case of the component $T_2$.
Then a direct application of Theorem~\ref{th_sys} yields the result.
\end{proof}
In the following example we show the applicability of Theorem~\ref{th_PDE_exis}.
\begin{example}
Take the open set $\Omega=\{(x,y)\in\mathbb{R}^2:x^2+y^2<1 \}$ and consider the system
	\begin{equation}\label{eq_ex}
		\left\{\begin{array}{ll} -\Delta u=\lambda_1(1+x^2)e^u(2+\cos v), & \text{ in } \Omega, \\[0.3cm] -\Delta v=\lambda_2(1+y^2)(1+v^2)(2+\sin u), & \text{ in } \Omega, \\[0.3cm] u=v=0, & \text{ on } \partial\,\Omega. \end{array}\right.
	\end{equation}
Now, note that the function 
$$
w(x,y):=\dfrac{1}{4}(1-x^2-y^2),
$$
is the unique solution of the ``torsion problem''
$$
-\Delta w=1\ \text{in}\ \Omega, \quad w=0\ \text{on}\ \partial\,\Omega. 
$$
Now fix $r_1, r_2>0$ and observe that the conditions $a)$ and $b)$ of Theorem~\ref{th_PDE_exis} are satisfied with the choice of $\underline{f}(x,y)=\underline{g}(x,y)\equiv 1$, since 
$$\sup_{(x,y)\in\overline{\Omega}}\int_{\Omega}k((x,y),(\xi,\eta))\,d(\xi,\eta)=\sup_{(x,y)\in\overline{\Omega}}\dfrac{1}{4}(1-x^2-y^2)=\dfrac{1}{4}.
$$
Note that $(r_1, r_2)$ can be chosen arbitrarily in $(0,+\infty)\times (0,+\infty)$, thus we obtain the existence of infinitely many couples of type $(\lambda_1, \lambda_2)$, with $\lambda_1,\lambda_2>0$, and associated couples of nonnegative functions~$(u_0, v_0)$ with prescribed norm that satisfy the system \eqref{eq_ex}.
\end{example}

\subsection{Eigenvalues for a system of ODEs} \label{eigen2}

Here we apply Theorem \ref{thmsys-kb} to the study of eigenvalues and eigenfunctions for the following class of BVPs for systems of ODEs:

\begin{equation}\label{syst_ode}
\left\{	\begin{array}{ll}
		u''(t)+\lambda_1 f(t,u(t),v(t))=0, \ & t\in[0,1], \\
		v''(t)+\lambda_2 g(t,u(t),v(t))=0, \ & t\in[0,1],\\
	    u'(0)=u(1)+u'(1)=0,  \\ v'(0)=v(1)-\frac12 v'(1)=0, 
	\end{array}
\right.
\end{equation}
where
$f:[0,1]\times \mathbb{R}_{+}\times\mathbb{R}\rightarrow\mathbb{R}_{+}$ and $g:[0,1]\times \mathbb{R}_{+}\times\mathbb{R}\rightarrow\mathbb{R}_{+}$ are continuous functions.

Note that the system~\eqref{syst_ode} can be rewritten as a system of Hammerstein integral equations, namely
\begin{equation}\label{eq_Ham}
\left\{
	\begin{array}{r}
		u(t)=\lambda_1 \displaystyle\int_{0}^{1} k_1(t,s)\,f(s,u(s),v(s))\,ds, \\[0.3cm]
		v(t)=\lambda_2 \displaystyle\int_{0}^{1} k_2(t,s)\,g(s,u(s),v(s))\,ds,
	\end{array}
	\right.
\end{equation}
where $k_1$ and $k_2$ are the corresponding Green's functions, which are given by
\[k_1(t,s)=\left\{\begin{array}{ll} {2-t}, & \quad s\leq t, \\ {2-s}, & \quad s>t, \end{array} \right. \]
and 
\[k_2(t,s)=\frac12 \left\{\begin{array}{ll} 1-2t, & \quad s\leq t, \\ 1-2s, & \quad s>t. \end{array} \right. \]
We point that the kernel $k_1$ is continuous, non-negative, and has been studied in details in~\cite{Lan, lan2009} while the kernel $k_2$ is continuous, changes sign, and can be obtained via a direct calculation.
%found in Hartman's book: see \cite[Chap.~XI]{hart}.
In this case we utilize the space $C[0,1]$, endowed with the usual supremum norm $\|u\|_{\infty}:=\max_{t\in [0,1]}|u(t)|$.
We construct, following the computations in~\cite{lan2009}, the cone
\begin{equation}\label{Guo-cone}
K = \left\{u\in C[0,1]:  \displaystyle{\min_{t\in[0,1]}}u(t) \geq \Bigl(1-\frac{t}{2}\Bigr) \|u\|_{\infty}\right\},
\end{equation}
and consider its product with the space itself.

With these ingredients we can state the following Theorem; the construction of the minorant $\underline{f}$ follows the argument in~\cite{lan2001}.

\begin{theorem}\label{th_ODE_exis}
Let $r_1,r_2$ be positive constants and suppose that exist two measurable functions $\underline{f}: [0,1]\rightarrow\mathbb{R}_+$
and $\underline{g}: [0,1]\rightarrow\mathbb{R}_+$ such that the following conditions hold:
	\begin{enumerate}[$a)$]
		\item $ f(t,u,v)\geq \underline{f}(t)$ on $[0,1]\times [r_1/2,r_1]\times[-r_2,r_2]$ and 
		\[\sup_{t \in [0,1]}  \displaystyle \int_{0}^{1}k_1(t,s)\, \underline{f}(s)\,ds >0;  \]
		\item $g(t,u,v)\geq \underline{g}(t)$ on $[0,1]\times [r_1/2,r_1]\times[-r_2,r_2]$
		and 
		\[\displaystyle\int_{0}^{1} \underline{g}(s)\,ds>0.  \]
	\end{enumerate}	

Then there exist $\lambda_{1,1},\lambda_{2,1},\lambda_{1,2} >0$, $\lambda_{2,2}<0$ and $(u_{0,j},v_{0,j})\in \overline{K}_{r_1}\times \overline{B}_{r_2} $, $j=1,2$, with $\left\|u_{0,j}\right\|=r_1$ and $\left\|v_{0,j}\right\|=r_2$ that satisfy the system~\eqref{eq_Ham}.
\end{theorem}

\begin{proof}
Let us consider the Banach spaces  $X_1=X_2=C[0,1]$ and the cone $K$ as in~\eqref{Guo-cone}.
Note that the operator
$$T=(T_1,T_2):\overline{K}_{r_1}\times \overline{B}_{r_2}\longrightarrow K\times C[0,1],$$
defined by
\begin{equation*}\label{eq_Ham0}
	\begin{array}{r}
		T_1(u,v)(t):=\displaystyle\int_{0}^{1} k_1(t,s)\,f(s,u(s),v(s))\,ds, \\[0.3cm]
		T_2(u,v)(t):=\displaystyle\int_{0}^{1} k_2(t,s)\,g(s,u(s),v(s))\,ds,
	\end{array}
\end{equation*}
is compact, as a consequence of the classical Arzel\`a--Ascoli Theorem. 

Firstly, let us take $(u,v)\in K\times C[0,1]$ such that $\left\|u\right\|_{\infty}=r_1, \left\|v\right\|_{\infty}= r_2$. Note that for every $t\in[0,1]$ we have $$u(t)\geq 
\Bigl(1-\frac{t}{2}\Bigr) r_1\geq \frac{r_1}{2}.$$ Therefore, we have
$$
\|T_1(u,v)\|_{\infty} \geq T_1(u,v)(t)= \displaystyle\int_{0}^{1} k_1(t,s)\,f(s,u(s),v(s))\,ds\geq \int_{0}^{1}k_1(t,s)\, \underline{f}(s)\,ds,
$$
for every $t\in[0,1]$.
Thus we obtain
\begin{equation}\label{RHS}
\|T_1(u,v)\|_{\infty} \geq \sup_{t \in [0,1]}  \displaystyle \int_{0}^{1}k_1(t,s) \, \underline{f}(s)\,ds.
\end{equation}
Note that the RHS of \eqref{RHS} does not depend on the particular $(u,v)$ chosen. Therefore we  obtain
$$
\inf_{\left\|u\right\|_{\infty}=r_1,\  \left\|v\right\|_{\infty}= r_2}\left\|T_1(u,v)\right\|_{\infty} \geq \sup_{t \in [0,1]}  \displaystyle \int_{0}^{1}k_1(t,s)\, \underline{f}(s)\,ds >0.
$$

Secondly, let us take $(u,v)\in K\times C[0,1]$ such that $\left\|u\right\|_{\infty}= r_1, \left\|v\right\|_{\infty}= r_2$. 
Note that
\begin{equation}\label{RHS2}
\begin{array}{l}
\|T_2(u,v)\|_{\infty} \geq  |T_2(u,v)(1)|\\ 
\phantom{\|T_2(u,v)\|_{\infty} \geq}= \displaystyle\int_{0}^{1} -k_2(1,s)\,g(s,u(s),v(s))\,ds \geq \displaystyle\int_{0}^{1} -k_2(1,s)\,\underline{g}(s)\,ds.
\end{array}
\end{equation}
Note that the RHS of \eqref{RHS2} does not depend on the particular $(u,v)$ chosen. Therefore we  obtain
$$
\inf_{\left\|u\right\|_{\infty}= r_1,\  \left\|v\right\|_{\infty}= r_2}\left\|T_2(u,v)\right\|_{\infty} \geq \displaystyle\int_{0}^{1} -k_2(1,s)\,\underline{g}(s)\,ds= \frac{1}{2}\displaystyle\int_{0}^{1} \underline{g}(s)\,ds>0.
$$
Then a direct application of Theorem~\ref{thmsys-kb} yields the result.
\end{proof}
In the following example we show the applicability of Theorem~\ref{th_ODE_exis}.
\begin{example}
Consider the system
\begin{equation*}\label{syst_ode_example}
\left\{	\begin{array}{l}
		u''(t)+\lambda_1 t(u^2+v^2)=0, \ t\in[0,1], \\
		v''(t)+\lambda_2 te^{uv}=0, \ t\in[0,1],\\
	       u'(0)=u(1)+u'(1)=0=v'(0)=v(1)-\frac12 v'(1).
	\end{array}
\right.
\end{equation*}
Now fix $r_1, r_2>0$, then the choice of $\underline{f}(t)= \dfrac{t \, r^2_1}{4}$ gives, by direct calculation,
$$\sup_{t \in [0,1]}  \displaystyle \int_{0}^{1}k_1(t,s) \frac{t \, r^2_1}{4}\,ds =\frac{r^2_1}{4}>0,$$
while choosing 
$\underline{g}(t)= te^{-r_1r_2}$ yields
$$
\displaystyle\int_{0}^{1}  te^{-r_1r_2}\,dt=\frac{e^{-r_1r_2}}{2}>0.
$$
Then the conditions $a)$ and $b)$  of Theorem~\ref{th_ODE_exis} are satisfied. Furthermore
since the pair $(r_1, r_2)$ can be chosen arbitrarily in $(0,+\infty)\times (0,+\infty)$,
 we obtain the existence of two distinct families of uncountably many pairs: both
$(\lambda_{1,1},\lambda_{2,1})$, with $\lambda_{i,1}>0$, $i=1,2$,
and
$(\lambda_{1,2},\lambda_{2,2})$, with $\lambda_{1,2}>0$ and $\lambda_{2,2}<0$,
each of them with the associated eigenfunctions of prescribed norms.
\end{example}

\section*{Acknowledgements}
The authors would like to thank the Referee the careful reading of the manuscript and the constructive comments.
A. Calamai and G.~Infante are members of the Gruppo Nazionale per l'Analisi Matematica, la Probabilit\`a e le loro Applicazioni (GNAMPA) of the Istituto Nazionale di Alta Matematica (INdAM).
G.~Infante is a member of the UMI Group TAA  and the ``The Research ITalian network on Approximation (RITA)''.

\section*{Funding}
J. Rodr\'iguez--L\'opez has been partially supported by the VIS Program of the University of Calabria, by Ministerio de Ciencia y Tecnología (Spain), AEI and Feder, grant PID2020-113275GB-I00, and by Xunta de Galicia, grant ED431C 2023/12.
This study was partly funded by: Research project of MIUR (Italian Ministry of Education, University and Research) Prin 2022 “Nonlinear differential problems with applications to real phenomena” (Grant Number: 2022ZXZTN2). 

\section*{Conflicts of interest}
The authors declare no conflict of interest.

\section*{Contribution statement}
All authors contributed equally to this manuscript.


\begin{thebibliography}{99}

\bibitem{ROW}
R. P. Agarwal, D. O'Regan and P. J. Y. Wong, \textit{Constant-sign solutions of systems of integral equations}, Springer, Cham,  2013.

\bibitem{ADV} J. Appell, E. De Pascale and A. Vignoli, \textit{Nonlinear spectral
theory}, Walter de Gruyter \& Co., Berlin,~2004.

\bibitem{Avramescu}	C. Avramescu, On a fixed point theorem, {\it St. Cerc. Mat.}, {\bf22(2)}, (1970) 215--221 (in Romanian). 

\bibitem{Benedetti}	
I. Benedetti, T. Cardinali and R. Precup, 
Fixed point-critical point hybrid theorems and application to systems with partial variational structure, {\it J. Fixed Point Theory Appl.}, {\bf 23} (2021), Paper No. 63, 19 pp.	

\bibitem{B-K-1922}
G. D. Birkhoff and O. D. Kellogg, Invariant points in function space, \textit{Trans. Amer. Math. Soc.}, \textbf{23} (1922), 96--115.

\bibitem{CCI14} A. Cabada, J. \'A. Cid and G. Infante, A positive fixed point theorem with applications to systems of Hammerstein integral equations, {\it Bound. Value Probl.} (2014) 2014:254.

\bibitem{CI_MMAS} A. Calamai and G. Infante, An affine Birkhoff--Kellogg-type result in cones with applications to functional differential equations, \textit{Math. Meth. Appl. Sci.}, \textbf{46} (2023), no. 11, 11897--11905.

\bibitem{chzh} X. Cheng and Z. Zhang, Positive solutions for a class of multi-parameter elliptic systems, \textit{Nonlinear Anal. Real World Appl.},
\textbf{14} (2013), 1551--1562.

\bibitem{cui3}
R. Cui, P.  Li, J. Shi and Y. Wang,
Existence, uniqueness and stability of positive solutions for a class of semilinear elliptic systems,
\textit{Topol. Methods Nonlinear Anal.},
\textbf{42} (2013), 91--104.

\bibitem{fel} G. Feltrin, A note on a fixed point theorem on topological cylinders, \textit{Ann. Mat. Pura Appl.}, \textbf{196} (2017), 1441--1458.

\bibitem{GT} 
D. Gilbarg and N.S. Trudinger,
\textit{Elliptic Partial Differential Equations of Second Order},
Springer-Verlag, Berlin, (2011).

\bibitem{GD} A. Granas and J. Dugundji, \textit{Fixed Point Theory},
Springer, New York, 2003.

\bibitem{guolak} D. Guo and V. Lakshmikantham,
\textit{Nonlinear problems in abstract cones}, Academic Press, Boston~1988. 

\bibitem{hai} D. D. Hai, Uniqueness of positive solutions for semilinear elliptic systems, \textit{J. Math. Anal. Appl.}, \textbf{313} (2006), no. 2, 761--767.

\bibitem{HL}
J. Henderson and R. Luca, \textit{Boundary value problems for systems of differential, difference and fractional equations. Positive solutions}, Elsevier, Amsterdam, 2016. 

\bibitem{gi-BK}
G. Infante, Eigenvalues of elliptic functional differential systems via a Birkhoff--Kellogg type theorem, \textit{Mathematics}, \textbf{9} (2021), n.~4.

\bibitem{imap} G. Infante, M. Maciejewski and R. Precup, A topological approach to the existence
and multiplicity of positive solutions of $(p, q)$-Laplacian systems, \textit{{Dyn. Partial Differ. Equ.}}, \textbf{12} 3 (2015), 193--215.

\bibitem{InMaRo} G. Infante, G. Mascali and J. Rodr\'iguez-L\'opez, A hybrid Krasnosel'ski\u{\i}-Schauder fixed point theorem for systems,
\textit{Nonlinear Anal. Real World Appl.}, \textbf{80} (2024), 1--9.

\bibitem{Krasno}
M. A. Krasnosel'ski\u{i}, \textit{Positive solutions of operator equations}, Noordhoff, Groningen, 1964.

\bibitem{Kra-Lady}
M. A. Krasnosel'ski\u{\i} and L. A. Lady\v{z}enski\u{\i}, The structure of the spectrum of positive nonhomogeneous operators, \textit{Trudy Moskov. Mat. Ob\v{s}\v{c}}, \textbf{3} (1954), 321--346.

\bibitem{Lan} K. Q. Lan, Multiple positive solutions of semilinear differential equations with singularities, \textit{J. London Math. Soc.}, \textbf{63} (2001), 690--704.

\bibitem{lan2001} K. Q. Lan, Eigenvalues of second order differential equations with singularities, \textit{Dynamical systems and differential equations (Kennesaw, GA, 2000)}, Discrete Contin. Dynam. Systems 2001, Added Volume, 241--247.

\bibitem{lan2009} K. Q. Lan, Eigenvalues of semi-positone Hammerstein integral equations and applications to boundary value problems, \textit{Nonlinear Anal.}, \textbf{71}    (2009), 5979--5993.

\bibitem{lan2} K. Q. Lan, Existence of nonzero positive solutions of systems of second order elliptic boundary value problems \textit{J. Appl. Anal. Comput.}, \textbf{1} (2011), 21--31.

\bibitem{lan2021} K. Q. Lan, Coexistence fixed point theorems in product Banach spaces and applications, \textit{Math. Meth. Appl. Sci.}, \textbf{44} (2021), 3960--3984.

\bibitem{ma2} R. Ma, R. Chen and Y. Lu, Positive solutions for a class of sublinear elliptic systems, \textit{Bound. Value Probl.},
\textbf{2014:28}, (2014), 15 pp.

\bibitem{Perov} A. I. Perov and A. V. Kibenko, On a certain general method for investigation of boundary value problems, \textit{Izv. Akad. Nauk SSSR}, \textbf{30} (1966), 249--264 (in Russian).
 
\bibitem{PrecupFPT} R. Precup, A vector version of Krasnosel'ski\u{\i}'s fixed point theorem in cones and positive periodic solutions of nonlinear systems, {\it J. Fixed Point Theory Appl.}, {\bf 2} (2007), 141--151.

\bibitem{yang2021} G. Yang and K. Q. Lan, Solutions and eigenvalues of Laplace's equations in bounded open subsets, \textit{Electron. J. Differential Equations}, \textbf{2021} (2021), Paper No. 87, 15 pp.

\end{thebibliography}
\end{document}